\documentclass[11pt]{article}

\usepackage[top=1in, left=1.50in, bottom=1in, right=1.50in]{geometry}

\usepackage{hyperref}  %provides clickable links to references within the document
\usepackage{setspace} 

\usepackage{graphicx}
\usepackage{amsfonts}
\usepackage{amssymb}
\usepackage{amsmath}
\usepackage{amsthm}

\newtheorem{theorem}{Theorem}[section]

\newtheorem{lemma}[theorem]{Lemma}
\newtheorem{conjecture}[theorem]{Conjecture}

\newcommand{\cl}{\hbox{\rm cl}}

\newcommand{\si}{\hbox{\rm si}}

\newcommand{\cB}{\mathcal{B}}
\newcommand{\cC}{\mathcal{C}}

\newcommand{\cA}{\mathcal{A}}

\title{On recognising frame and lifted-graphic matroids}

\author{Rong Chen$^1$, Geoff Whittle$^2$\\
\\
$^1$Center for Discrete Mathematics,\ \ Fuzhou University\\
Fuzhou,\ \ P. R. China\\
$^2$School of Mathematics,\ \ Statistics and Operations Research\\ 
Victoria University of Wellington, New Zealand.}

\begin{document}

\maketitle

\footnote{ 

Emails: rongchen@fzu.edu.cn (R. Chen),\ \ geoff.whittle@uvw.ac.nz (G. Whittle). 

The projected is supported partially by NSFC (No. 11471076) and a grant from the Marsden Fund of New Zealand.}

\begin{abstract}
We prove that there is no polynomial $p(\cdot)$ with the property that a matroid $M$ can be
determined to be either a lifted-graphic or frame matroid using at most $p(|M|)$
rank evaluations. This resolves two conjectures 
of Geelen, Gerards and Whittle (Quasi-graphic 
matroids, arXiv:1512.03005v1).

\end{abstract}

\section{Introduction}

In \cite{Geelen}, Geelen, Gerards and Whittle make the following conjectures.

\begin{conjecture}\label{Jim-1}
For any polynomial $p(\cdot)$ there is a frame matroid $M$ such that for any set $\cA$ 
of subsets of $E(M)$ with $|\cA|\leq p(|E(M)|)$ there is a non-frame matroid $M'$ such 
that $E(M')=E(M)$ and $r_{M'}(A)=r_M(A)$ for each $A\in\cA$. 
\end{conjecture}

\begin{conjecture}\label{Jim-2}
For any polynomial $p(\cdot)$ there is a lifted-graphic matroid $M$ such that for any 
set $\cA$ of subsets of $E(M)$ with $|\cA|\leq p(|E(M)|)$ there is a non-lifted-graphic 
matroid $M'$ such that $E(M')=E(M)$ and $r_{M'}(A)=r_M(A)$ for each $A\in\cA$. 
\end{conjecture}

Put in another way, the conjectures assert that, for matroids given by rank oracles, there
does not exist a polynomial-time algorithm that determines whether a matroid is a frame
matroid or a lifted-graphic matroid. In this paper we give proofs of both of these
conjectures. In fact we prove slightly stronger results, in that we resolve the above
conjectures within the class of quasi-graphic matroids. 

We devote the remainder of the introduction to clarifying the notions considered above.
We begin by recalling material from \cite{Geelen}. 

Let $G$ be a graph and let $M$ be a matroid. For a vertex $v$ of $G$ we let loops$_G(v)$
denote the set of loop-edges of $G$ at the vertex $v$. We say that $G$ is a 
{\em framework} for $M$ if
\begin{enumerate}
\item $E(G)=E(M)$,
\item $r(E(H))\leq |V(H)|$ for each component $H$ of $G$, and
\item for each vertex $v$ of $G$ we have 
$\cl_M(E(G-v))\subseteq  E(G-v)\cup {\rm loops}_G(v)$.
\end{enumerate}
A matroid is {\em quasi-graphic} if it has a framework. A matroid $M$ is a {\em lifted-graphic}
matroid if there is a matroid $M'$ and an element $e\in E(M')$ such that $M'\backslash e=M$ and
$M/e$ is graphic. A matroid $M$ is {\em framed} if it has a basis $V$ such that,
for each element $e\in E(M)$, there is a subset $W$ of $V$ with at most two elements
such that $e\in\cl_M(W)$. A {\em frame matroid} is a restriction of a framed matroid.
It is proved in \cite{Geelen} that frame matroids and lifted-graphic matroids
are quasi-graphic. It is also proved in \cite{Geelen} that if $M$ is a 
representable, 3-connected, quasi-graphic matroid, then $M$ is either a frame matroid
or a lifted-graphic matroid. However, for non-representable matroids this is very 
far from the case. 

Frame matroids and lifted-graphic matroids were introduced by Zaslavsky 
\cite{Zaslavsky-II} from the perspective of biased graphs \cite{Zaslavsky-I}.
We will need that perspective for this paper so we recall material from
\cite{Zaslavsky-I,Zaslavsky-II} now. A {\em theta graph} is a graph that consists of a 
pair of vertices joined by three internally disjoint paths.  A connected 2-regular
graph is a {\em cycle}. A collection ${\mathcal B}$
of cycles of a graph $G$ satisfies the {\em theta property} of no theta subgraph of 
$G$ contains exactly two members of $\mathcal B$. 
A \emph{biased graph} consists of a pair $(G, \mathcal{B})$, where 
$G$ is a graph and $\mathcal{B}$ is a collection of cycles of $G$ satisfying the theta
property. If $(G,{\mathcal B})$ is a biased graph, then the members of 
${\mathcal B}$ are called {\em balanced cycles}, otherwise cycles of $G$ are called 
{\em unbalanced}.

For a set $A$ of edges of a graph $G$, let $G[A]$ be the subgraph with edge set $A$ and
vertex set consisting of all vertices incident with an edge in $A$. Zaslavsky 
\cite{Zaslavsky-II} defines two matroids associated with  
a biased graph $(G,\mathcal B)$.  In the first,
denoted LM$(G,\mathcal B)$, a subset $I$ of $E(G)$ is independent if and only if 
$G[I]$ contains no balanced cycle and at most one cycle. 
In the second, denoted FM$(G,\mathcal B)$, a subset
$I$ of $E(G)$ is independent if and only if $G[I]$ contains no balanced cycle
and every component of $G[I]$ contains at most one cycle. The next theorem follows from
work of Zaslavsky \cite{Zaslavsky-II}.

\begin{theorem}
\label{zas}
Let $M$ be a matroid.
\begin{itemize}
\item[(i)] $M$ is a lifted-graphic matroid if and only if there exists a biased graph 
$(G,\mathcal B)$ such that $M={\rm LM}(G,\mathcal B)$.
\item[(ii)] $M$ is a frame-matroid if and only if there exists a biased graph $(G,\mathcal B)$
such that  $M=FM(G,\mathcal B)$.
\end{itemize}
\end{theorem}

Finally we note that, if $M={\rm LM}(G,\mathcal B)$ or $M={\rm FM}(G,\mathcal B)$ for
some biased graph $(G,\mathcal B)$, then $G$ is a framework for $M$.

%\begin{conjecture}\label{Jim-1}
%There is no polynomial-time algorithm that, given a matroid $M$ via its rank oracle and a graph $G$, determines whether there is a biased graph $(G,\mathcal{B})$ such that $M=FM(G,\mathcal{B})$.
%\end{conjecture}

%\begin{conjecture}\label{Jim-2}
%There is no polynomial-time algorithm that, given a matroid $M$ via its rank oracle and a graph $G$, determines whether there is a biased graph $(G,\mathcal{B})$ such that $M=LM(G,\mathcal{B})$.
%\end{conjecture}

In \cite{Geelen} it is also conjectured that, unlike lifted-graphic and frame matroids,
the property of being a quasi-graphic matroid can be recognised with a polynomial number
of rank evaluations. Given the results of this paper it is clear that this conjecture
is the more natural extension of a theorem of Seymour \cite{Seymour} where he proves that
graphic matroids can be recognised with a polynomial number of rank evaluations.

\section{Relaxations and tightenings} 

Recall that a {\em circuit-hyperplane} of a matroid $M$ is a set $C$ that is both a circuit
and a hyperplane. It is well known, see for example \cite[Proposition~1.5.14]{Oxley},
that if $C$ is a circuit-hyperplane of $M$, then 
$\mathcal B(M)\cup \{C\}$ is the set of bases of a matroid $M'$. In this case we say that
$M'$ is obtained from $M$ by {\em relaxing} the circuit-hyperplane $C$. 

Relaxation will be an important operation for us, but we will also need the reverse 
operation, which is no doubt well understood, but does not seem to appear in the literature.

Let $B$ be a basis of a matroid $M$. If the closure of each proper subset of $B$ is itself, 
then we say that $B$ is {\sl free}. Observe that, if $B$ is a free basis, then 
for each $e\in B$ and $f\in E(M)-B$ the set $B-e+f$ is a basis of $M$. 

\begin{lemma}\label{intensify}
Let $B$ be a free basis of a matroid $M$. Then $\mathcal{B}(M)-\{B\}$ is the set of bases of a matroid. 
\end{lemma}

\begin{proof}
Let $B_1,B_2\in\mathcal{B}(M)-\{B\}$ and $e\in B_1-B_2$. Then there is $f\in B_2-B_1$ satisfying $B_1-\{e\}+\{f\}\in\cB(M)$. Assume that $B=B_1-\{e\}+\{f\}$. Then $\cl_M(B_1-\{e\})=B_1-\{e\}$. Since $B_2-\{f\}\neq B_1-\{e\}$ as $B_2\neq B$, there is an element $f'\in B_2-(B_1\cup \{f\})$. Moreover, since $\cl_M(B_1-\{e\})=B_1-\{e\}$, we have $B_1-\{e\}+\{f'\}\in\cB(M)-\{B\}$. 
%Evidently, it suffices to show that for each circuit $C$ of $M$ with $B\nsubseteq C$ if $e\in C\cap B$ then the set $(B\cup C)-\{e\}$ contains a circuit of $M$. Since the closure of each proper subset of $B$ is itself and $B\nsubseteq C$, we have $|C-B|\geq2$, say $f, f'\in C-B$. Since $B-e+f$ is a basis of $M$, we have $f'\in\cl_M(B-e+f)$. So the set $(B\cup C)-\{e\}$ contains a circuit of $M$. 
\end{proof}

%\begin{lemma}\label{intensify}
%Let $B$ be a free basis of a matroid $M$. Then $\{B\}\cup\{C\in\mathcal{C}(M)| B\nsubseteq C\}$ is the set of circuits of some matroid. 
%\end{lemma}

%\begin{proof}
%Evidently, it suffices to show that for each circuit $C$ of $M$ with $B\nsubseteq C$ if $e\in C\cap B$ then the set $(B\cup C)-\{e\}$ contains a circuit of $M$. Since the closure of each proper subset of $B$ is itself and $B\nsubseteq C$, we have $|C-B|\geq2$, say $f, f'\in C-B$. Since $B-e+f$ is a basis of $M$, we have $f'\in\cl_M(B-e+f)$. So the set $(B\cup C)-\{e\}$ contains a circuit of $M$. 
%\end{proof} 

We say that the matroid $(E(M), \mathcal{B}(M)-\{B\})$ given in Lemma \ref{intensify} 
is obtained from $M$ by {\sl tightening the free basis} $B$. Evidently, tightening is 
the reverse operation of relaxation. The following results are obvious and will be 
used in the next section without reference. 

\begin{lemma}
Let $M'$ be a matroid obtained from a matroid $M$ by tightening a free basis $B$ . Then $r_{M'}(X)=r_M(X)$ for each $B\neq X\subseteq E(M)$. 
\end{lemma}

\begin{lemma}
Let $M'$ be a matroid obtained from a matroid $M$ by relaxing a circuit-hyperplane $C$ . Then $r_{M'}(X)=r_M(X)$ for each $C\neq X\subseteq E(M)$. 
\end{lemma}

%Therefore, for a matroid $M$ with many basis-circuits, Lemmas \ref{intensify} and \ref{keep-basis-circuit} imply that by a series of intensifying operations we can get a new matroid from $M$. For circuit-hyperplanes, we have a similar result as Lemma \ref{keep-basis-circuit}. Hence, for a matroid with many circuit-hyperplanes, by a series of relaxing operations we can also get a new matroid from the original one. In the proofs of Conjectures \ref{Jim-1} and \ref{Jim-2}, we will construct a class of biased graphs $(G,\emptyset)$ such that their frame matroids have many basis-circuits and their lift matroids have many circuit-hyperplanes. 

\section{Proof of the main theorems} 

In this section we give proofs of the following theorems, which, as observed earlier,
are slight strengthenings of Conjectures~\ref{Jim-1} and \ref{Jim-2}.

\begin{theorem}
\label{1}
For any polynomial $p(\cdot)$ there is a frame matroid $M$ such that for any set $\cA$ 
of subsets of $E(M)$ with $|\cA|\leq p(|E(M)|)$ there is a quasi-graphic
non-frame matroid $M'$ such 
that $E(M')=E(M)$ and $r_{M'}(A)=r_M(A)$ for each $A\in\cA$. 
\end{theorem}

\begin{theorem}
\label{2}
For any polynomial $p(\cdot)$ there is a lifted-graphic matroid $M$ such that for any 
set $\cA$ of subsets of $E(M)$ with $|\cA|\leq p(|E(M)|)$ there is a 
quasi-graphic non-lifted-graphic 
matroid $M'$ such that $E(M')=E(M)$ and $r_{M'}(A)=r_M(A)$ for each $A\in\cA$. 
\end{theorem}

Let $G$ be a graph. Recall that a cycle $C$ is {\em chordless} if there is no
edge in $E(G)-C$ that joins vertices of $C$. Two cycles are {\em disjoint}
if they do not share a common vertex. We say that the pair $C$, $C'$ of cycles is a 
{\em covering pair} if every vertex of $G$ is in either $C$ or $C'$. Our interest
will focus on covering pairs of disjoint chordless cycles. 

For the remainder of this paper we focus on biased graphs all of whose cycles are 
unbalanced, that is, biased graphs of the form $(G,\emptyset)$. 
We omit the elementary proofs of the next two lemmas. 

\begin{lemma}
\label{easy1}
Let $G$ be a graph and let 
$\{C,C'\}$ be a covering pair of disjoint chordless cycles of $G$. Then the following hold.
\begin{itemize}
\item[(i)] $C\cup C'$ is a circuit-hyperplane of 
the lift matroid ${\rm LM}(G,\emptyset)$.
\item[(ii)] $C\cup C'$ is a free basis of 
the frame matroid ${\rm FM}(G,\emptyset)$. 
\end{itemize}
\end{lemma}

\begin{lemma}
\label{easy2}
Let $G$ be a graph and let $\{C,C'\}$ be a covering pair of disjoint chordless cycles of
$G$. Then the following hold.
\begin{itemize}
\item[(i)] $G$ is a framework for the matroid obtained by relaxing the circuit-hyperplane
$C\cup C'$ of ${\rm LM}(G,\emptyset)$.
\item[(ii)] $G$ is a framework for the matroid obtained by tightening the free basis
$C\cup C'$ of ${\rm FM}(G,\emptyset)$.
\end{itemize}
\end{lemma}

Next we define the family of graphs that we will use to prove Theorems~\ref{1} and \ref{2}.

Let $n\geq4$ be an even number. Let $G_n$ be the graph with 
$V(G_n)=\{u_1,u_2,\ldots,u_n,\ v_1,\\ v_2,\ldots, v_n\}$ and 
$E(G_n)=E(C_1\cup C_2\cup C_3\cup C_4)$, where 
\[\begin{aligned}
C_1&=u_1u_2\ldots u_nu_1,\\ 
C_2&=v_1v_2\ldots v_nv_1,\\
C_3&=u_1v_1u_3v_3u_5\ldots v_{n-3}u_{n-1}v_{n-1}u_1,\\
C_4&=u_2v_2u_4v_4u_6\ldots v_{n-2}u_{n}v_{n}u_2.
\end{aligned}\]
are cycles of $G_n$. See Figure 1.  %Evidently, 
%\begin{itemize}
%\item[(3)] All theta-subgraphs and hancuffs of $G$ contain at least seven edges. 
%\end{itemize}
For each integer $1\leq i\leq n$, set $e_i=u_iv_i$ and $f_{i,i+2}=v_iu_{i+2}$, where the subscripts are modulo $n$. 
\begin{figure}[htbp]
\begin{center}
\includegraphics[page=1,height=4.5cm]{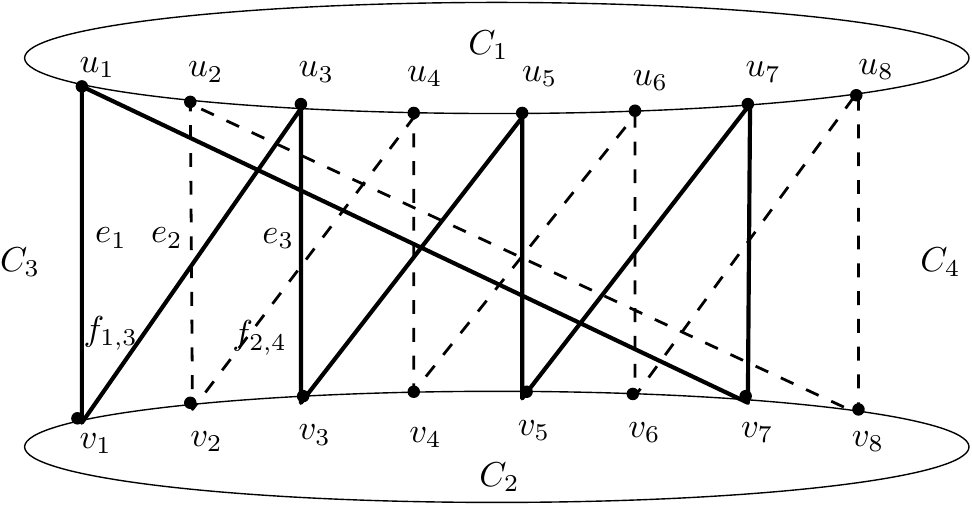}
\caption{the graph $G_8$.}
%\label{Figure 1}
\end{center}
\end{figure}
We say that $(e_{i+1}, f_{i, i+2})$ is a {\sl crossing pair}. 
Note that the two edges contained in a crossing pair are disjoint and each edge in 
$C_3\cup C_4$ belongs to exactly one crossing pair. 
Let $X$ be the set whose elements are the crossing pairs of $G_n$. Let $X'$ be the set 
of all subsets of $X$ with even size. Observe that  $|X|=n$, we have $|X'|=2^{n-1}$. 
%For any $S\in CP'$  there is  a unique pair of disjoint cycles $C_S^1$ and $C_S^2$ determined by $S$ with $E(S)\subseteq E(C_S^1\cup C_S^2)$. 
Let $$S=\{(e_{i_1+1}, f_{i_1, i_1+2}), (e_{i_2+1}, f_{i_2, i_2+2}),\ldots, (e_{i_{2k}+1}, f_{i_{2k}, i_{2k}+2})\}$$ 
be an element in $X'$ with $1\leq i_1<i_2<\ldots<i_{2k}\leq n$. 
There is  a unique pair of disjoint cycles $C_S^1$ and $C_S^2$ 
such that the set of edges in crossing pairs in $S$
is equal to $E(C_3\cup C_4)\cap E(C_S^1\cup C_S^2)$ and with 
\[\begin{aligned}
\{e_{i_1+1}, f_{i_2,i_2+2}, e_{i_3+1}, \ldots, f_{i_{2k}, i_{2k}+2}\}\subset C_S^1, \\
\{f_{i_1,i_1+2}, e_{i_2+1}, f_{i_3,i_3+2}, \ldots, e_{i_{2k}+1}\}\subset C_S^2.
\end{aligned}\]
See Figure 2. 
\begin{figure}[htbp]
\begin{center}
\includegraphics[page=2,height=4.5cm]{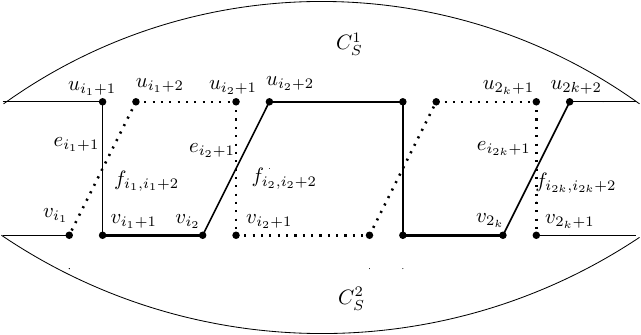}
\caption{the disjoint cycles $C_S^1$ and $C_S^2$.}
%\label{Figure 2}
\end{center}
\end{figure}
For example, when $S$ contains all crossing pairs, $C_S^1=C_3, C_S^2=C_4$. 
When $S=\emptyset$, 
we have $C_S^1=C_1, C_S^2=C_2$. 

It follows from routine inspection that,
for each $S\in X'$, the pair $\{C_S^1,C_S^2\}$ is a covering pair of
disjoint chordless cycles of $G_n$ and $C_S^1,C_S^2$ are of the same length $n$. 
Let $\mathcal Z$ be the set of edge sets of 
all covering pairs of disjoint chordless cycles obtained from $X'$ in the above way.
Then $|\mathcal Z|=2^{n-1}$. By Lemma~\ref{easy2} each member of $\mathcal Z$ is a circuit
hyperplane of LM$(G_n,\emptyset)$ and a free basis of FM$(G_n,\emptyset)$. 

To distinguish $LM(G_n,\emptyset)$ from the matroid obtained by relaxing a member $Z$ of
$\mathcal Z$ we need to check the rank of $Z$ and to distinguish 
$FM(G_n,\emptyset)$ from the matroid obtained by tightening $Z$ we also need to check
the rank of $Z$. To distinguish LM$(G_n,\emptyset)$ or FM$(G_n,\emptyset)$ from all such
matroids we need to check the rank of $2^{n-1}$ subsets. Evidently 
$2^{n-1}$ outgrows any polynomial function.

%Evidently, $\{C_S^1, C_S^2\}$  of $G$. Let $\cF$ be the set consisting of all such spanning pair of disjoint-cycles flat of $G$. Since $|CP'|=2^{n-1}$, we have $|\cF|=2^{n-1}$. 
%\begin{itemize}
%\item[(6)] For each spanning pair of disjoint cycles-hyperplane $X_S$ determined by $S$, 
%\end{itemize}

%Let $\cA$ be a set of subsets of $E(G)$ with $|\cA|\leq p(|E(G)|)=p(4n)$, where $p(\cdotp)$ is a polynomial function. Let $\cH=\cF-\cA$. Since $|\cF|=2^{n-1}$, we have $\cH\neq\emptyset$. Let $M_F$ be the matroid obtained from $FM(G,\emptyset)$ by tightingen all elements in $\cH$, and $M_L$ be the matroid obtained from $LM(G,\emptyset)$ by relaxing all elements in $\cH$. Evidently, 
%\[\begin{aligned}
%r(M_F)&=r(FM(G,\emptyset)),\ \ \ r(M_L)=r(LM(G,\emptyset)), \\
%r_{M_F}(A)&=r_{FM(G,\emptyset)}(A),\ r_{M_L}(A)=r_{LM(G,\emptyset)}(A)\ \text{for\ each}\ A\in\cA. 
%\end{aligned}\]

Say $Z\in \mathcal Z$. Let $M_L=LM(G_n,\emptyset)$ and let $M_F=FM(G_n,\emptyset)$.
Let $M_L^Z$ and $M_F^Z$ be the matroids obtained from $M_L$ and $M_F$ by respectively
relaxing and tightening $Z$. To complete the proof of Theorems~\ref{1} and \ref{2}
it suffices to show that $M_L^Z$ is not a lift matroid and $M_F^Z$ is not a frame matroid.
We now turn attention to this task. 

%Hence, to finish the proof of Conjectures , it suffices to show that $M_F$ is a non-frame matroid and $M_L$ is a non-lifted-graphic matroid. Next, we prove this.  

%Since each vertex of $G$ is adjacent with exactly two edges of a circuit-hyperplane of $G$, it is easy to prove that the graph $G$ is also a framework of $M_F$ and $M_L$. To determine whether $M_F=FM(G,\emptyset)$ or $M_L=LM(G,\emptyset)$, we must check the rank of each circuit-hyperplane of $G$. Since there are $2^{n-1}$ circuit-hyperplanes of $G$, there is no polynomial-time algorithm to determine whether $M_F=FM(G,\emptyset)$ or $M_L=LM(G,\emptyset)$. 

%Since each cycle of $G$ is independent in $M_F$ or $M_L$, by (3)  we have 

%\begin{lemma}\label{size}
%Each circuit of $M_F$ or $M_L$ has at least seven elements. 
%\end{lemma}

%\begin{lemma}\label{gh-1}
%Let $F\in\cF$ and $g,h\in E(G)-F$. Then there are exactly two circuits of $M_F|F\cup\{g,h\}$ or $M_L|F\cup\{g,h\}$ containing $\{g,h\}$. 
%\end{lemma}

%Since $G|E(G)-F$ is a 2-regular graph, by (4) and Lemma \ref{gh-1} we have 

Each cycle of $G_n$ in independent in $M_L^Z$ and $M_F^Z$. The next lemma
follows from this observation and inspection of the graph $G_n$. 

\begin{lemma}\label{add g and h}
\label{structure}
Let $g$ and $h$ be distinct edges in $E(G_n)-Z$.
\begin{itemize}
\item[(i)]  If $g$ and $h$ are adjacent in $G_n$, there is a partition $(P_1,P_2,P_3)$ of 
$Z$ with $|P_1|=2$ with $|P_1\cup P_2|=|P_3|=n$, and such that 
$(Z-P_i)\cup\{g,h\}$ is a circuit of $M_L^Z$ and $M_F^Z$ for each $1\leq i\leq 3$. 
%and $P_1\cup P_3\cup\{g,h\}$ is a tight handcuff of $G$; 
\item[(ii)] If $g$ and $h$ are not adjacent in $G_n$, then there is a partition 
$(P_1,P_2,P_3,P_4)$ of $Z$ with $|P_1\cup P_2|=|P_3\cup P_4|=n$ such that 
$(Z-P_i)\cup\{g,h\}$ is a circuit of $M_L^Z$ and $M_F^Z$ for each $1\leq i\leq 4$;
\item[(iii)] except the circuits in (a) and (b), there is no other circuit $C$ of 
$M_L^Z$ or $M_F^Z$ satisfying $\{g,h\}\subseteq C\subseteq Z\cup\{g,h\}$. 
\end{itemize}
\end{lemma}

Since $G_n[E(G_n)-Z]$ is a 2-regular graph, by Lemma~\ref{structure}, we have 

\begin{lemma}\label{gh-2}
There are exactly $2n$ pairs of edges $g, h$ in $E(G_n)-Z$ such that there is 
a partition $(P_1,P_2,P_3)$ of $Z$ with $|P_1|=2$ with $|P_1\cup P_2|=|P_3|=n$, 
and such that $(F-P_i)\cup\{g,h\}$ is a circuit of $M_L^Z$ and $M_F^Z$ 
for each $1\leq i\leq 3$
\end{lemma}

\begin{lemma}\label{G'}
Let $G'$ be a framework for $M_F^Z$ or $M_L^Z$. Then $G'$ is a $4$-regular 
graph with $2n$ vertices and without loops.
\end{lemma}

\begin{proof}
Evidently, $|V(G')|=2n$. Since each cocircuit in $M_F^Z$ or $M_L^Z$ has at least four elements 
and $|E(G')|=4n$, the graph $G'$ is a $4$-regular graph without loops. 
\end{proof}

Let $C$ be an even cycle of a graph $H$, and $e=uv$ be an edge in 
$E(H)-E(C)$ with $u,v\in V(C)$. If the two paths in $C$ joining $u,v$ 
have the same length, then $e$ is a {\sl bisector} of $C$. 

\begin{lemma}\label{G'|F-F}
Let $(G',\cB')$ be a biased graph satisfying $M_F^Z=FM(G',\cB')$. Then $G'[Z]$ is not a cycle. 
\end{lemma}

\begin{proof}
Assume to the contrary that $Z$ is a cycle of $G'$. Since $|Z|=2n$, we have $V(G')=V(G'[Z])$. Since $Z\in\cC(M_F^Z)$, we have $Z\in\cB'$. Let $g,h\in E(G_n)-Z$. since $Z$ is a circuit-hyperplane of $M_F^Z$ and $G'[Z]$ is a cycle in $\cB'$, except $Z$ each cycle in $G'[Z\cup\{g\}]$ and $G'[Z\cup\{h\}]$ are not in $\cB'$. When $g,h$ are adjacent in $G'$, not assuming $g$ and $h$ adjacent in $G_n$ by Lemma \ref{add g and h} the unique cycle containing $\{g,h\}$ in $G'[Z\cup\{g,h\}]$ is not in $\cB'$; hence, there is a partition $(P_1,P_2,P_3)$ of $Z$ such that $(Z-P_i)\cup\{g,h\}$ is a circuit of $M_F^Z$ for each $1\leq i\leq 3$. Moreover, since $G'[E(G_n)-Z]$ is a 2-regular graph with exactly $2n$ vertices, by Lemma \ref{gh-2}, we have that (a) when $g,h$ are not adjacent in $G'$ there is a partition $(P_1,P_2,P_3,P_4)$ of $Z$ with $|P_1\cup P_2|=|P_3\cup P_4|=n$ and such that $(Z-P_i)\cup\{g,h\}$ is a circuit of $M_F^Z$ for each $1\leq i\leq 4$; and (b) for each $v\in V(G')$, assuming that $\{e,e'\}$ is the set of edges adjacent with $v$ but not in $Z$, the unique cycle $C$ containing $\{e,e'\}$ in $G'[Z\cup\{e,e'\}]$ has exactly four edges and either $e$ or $e'$ is a bisector of the cycle $Z$. Assume that $e'$ is a bisector of $Z$ and $u$ is the unique vertex in $C$ not adjacent with $e$ or $e'$. Let $f$ be the bisector of $Z$ adjacent with $u$. By the arbitrary choice of $v$ and (b) such $f$ exists. Since $\si(G'[Z\cup\{e,f\}])$ is a 4-edge cycle and $f$ is diameter of $Z$, (a) can not hold,  a contradiction. 
\end{proof}

Similarly (in fact, in a simpler way), we can prove 

\begin{lemma}\label{G'|F-L}
Let $(G',\cB')$ be a biased graph satisfying $M_L^Z=LM(G',\cB')$. Then $G'[Z]$ is not a cycle. 
\end{lemma}

\begin{lemma}\label{non-frame}
$M_F^Z$ is a non-frame matroid. 
\end{lemma}

\begin{proof}
Assume to the contrary that there is a biased graph $(G',\cB')$ such that $M_F^Z=FM(G',\cB')$.  Since $Z\in\cC(M_F^Z)$, by Lemma \ref{G'|F-F} the graph $G'[Z]$ is a theta-graph or a handcuff. Hence, $|V(G'[Z])|=2n-1$ as  $|Z|=2n$. Moreover, since $Z$ is a circuit-hyperplane of $M_F^Z$, we have $g\notin\cl_{M_F}(Z)$ for each $g\in E(G_n)-Z$; so all edges in $E(G_n)-Z$ are adjacent with the unique vertex in $V(G')-V(G'[Z])$. Hence, $G'$ is not 4-regular as $|E(G_n)-Z|=2n\geq8$, a contradiction to Lemma \ref{G'}. 
\end{proof} 

\begin{lemma}\label{non-lifted}
$M_L^Z$ is a non-lifted-graphic matroid. 
\end{lemma}

\begin{proof}
Assume to the contrary that there is a biased graph $(G',\cB')$ satisfying $M_L^Z=LM(G',\cB')$. Since $Z$ is a basis of $M_L^Z$ and $G'[Z]$ is not a cycle by Lemma \ref{G'|F-L}, the graph $G'[Z]$ has degree-1 vertices and all edges in $E(G')-Z$ must be adjacent with all degree-1 vertices in $G'[Z]$, so $G'$ is not 4-regular, a contradiction to Lemma \ref{G'}. 
\end{proof} 

Theorems~\ref{1} and \ref{2} now follow from the fact that $G_n$ has exponentially
many covering pairs of disjoint chordless cycles, that the matroids obtained by
relaxing the circuit-hyperplane of LM$(G_n,\emptyset)$ or tightening the free basis
of FM$(G_n,\emptyset)$ associated with any one of these chordless cycles is quasi-graphis,
and Lemmas~\ref{non-lifted} and \ref{non-frame} which show that these matroids are 
respectively not lifted-graphic and not frame matroids.

\end{document}